\documentclass[12pt, reqno]{amsart} 
\usepackage{amsmath, amsthm, amscd, amsfonts, amssymb, graphicx, color}
\usepackage[bookmarksnumbered, colorlinks, plainpages]{hyperref}
\hypersetup{colorlinks=true,linkcolor=red, anchorcolor=green, citecolor=cyan, urlcolor=red, filecolor=magenta, pdftoolbar=true}

\newtheorem{theorem}{Theorem}[section]

\newtheorem{cor}[theorem]{Corollary}
\theoremstyle{definition}

\theoremstyle{remark}
\newtheorem{remark}[theorem]{\bf{Remark}}
\numberwithin{equation}{section}
\begin{document}

\title []    {{  Numerical radius inequalities of bounded linear operators and  $(\alpha,\beta)$-normal operators  }}


\author[P. Bhunia]{Pintu Bhunia}


\address{Department of Mathematics, Indian Institute of Science, Bengaluru 560012, Karnataka, India}
\email{pintubhunia5206@gmail.com}


\thanks{The author  would like to sincerely acknowledge Prof. Kallol Paul for his valuable comments on this paper.}
	\thanks{ The author also would like to thank SERB, Govt. of India for the financial support in the form of National Post Doctoral Fellowship (N-PDF, File No. PDF/2022/000325) under the mentorship of Prof. Apoorva Khare. }

\thanks{}


\maketitle

\begin{abstract}
We obtain various upper bounds for the numerical radius $w(T)$ of a bounded linear operator $T$ defined on a complex Hilbert space $\mathcal{H}$, by developing the upper bounds for the $\alpha$-norm of $T$, which is defined as  $\|T\|_{\alpha}= \sup \left\{ \sqrt{\alpha |\langle Tx,x \rangle|^2+ (1-\alpha)\|Tx\|^2 } : x\in \mathcal{H}, \|x\|=1 \right\}$ for $ 0\leq \alpha \leq 1 $. Further, we prove that  
	\begin{eqnarray*}
		w(T) &\leq & \sqrt{\left( \min_{\alpha \in [0,1]}\left\| \alpha |T|+(1-\alpha)|T^*| \right\| \right) \|T\|} \,\,\,\, \leq \,\, \,\, \|T\|.
	\end{eqnarray*} 
	  For $0\leq \alpha \leq 1 \leq \beta,$ the operator $T$ is called $(\alpha,\beta)$-normal if $\alpha^2 T^*T\leq TT^*\leq \beta^2 T^*T$ holds. Note that every invertible operator is an $(\alpha,\beta)$-normal operator for suitable values of $\alpha$ and $\beta$.  Among other lower bound for the numerical radius of an $(\alpha,\beta)$-normal operator $T$, we show that 
\begin{eqnarray*}
	w(T)	&\geq & \sqrt{\max \left\{ 1+\alpha^2, 1+\frac{1}{\beta^2}\right\} \frac{\|T\|^2}{4}+ \frac {\left| \|\Re(T)\|^2-\|\Im(T)\|^2  \right|}2} \\
	&\geq & \max \left\{ \sqrt{1+\alpha^2}, \sqrt{1+\frac{1}{\beta^2}} \right\} \frac{\|T\|}{2}  \\
	 &	> & \frac{\|T\|}2,
\end{eqnarray*}
where $\Re(T)$ and $\Im(T)$ are the real part and imaginary part of $T$, respectively.

\end{abstract}

\section{\textbf{Introduction}}
\noindent
Let $\mathcal{B}(\mathcal{H})$ denote the $C^{*}$-algebra of all bounded linear operators on a complex Hilbert space $\mathcal{H}$, with inner product $\langle \cdot, \cdot \rangle.$  Let $T \in \mathcal{B}(\mathcal{H}).$ We write $|T|=(T^*T)^{1/2}$ and $|T^*|=(TT^*)^{1/2}$, where $T^*$ denotes the adjoint of $T.$
The numerical radius of $T$, denoted by $w(T)$, is defined as 
  $$ w(T)=\sup \{|\langle Tx,x \rangle| : x\in \mathcal{H}, \|x\|=1  \}.$$
Note that, the numerical radius defines a norm on  $\mathcal{B}(\mathcal{H})$, and it is equivalent to the operator norm, satisfies  
\begin{eqnarray}\label{eqv1}
	\frac12 \|T\| \leq w(T) \leq \|T\|, \,\mbox{ for every $T \in  \mathcal{B}(\mathcal{H}).$}
\end{eqnarray}
Various refinements of the inequalities in \eqref{eqv1} and related results have been discussed in a recent published book \cite{book}. Also, the reader can see the papers \cite{Bag_MIA_2020,Bhunia_LAA_2021, Kittaneh_STD_2005} and references therein.
 For $0\leq \alpha \leq 1$, the $\alpha$-norm of $T$ (see \cite{Bhunia_MMN_2022,Bhunia_JCA_2022,Sain_AFA_2021}) is given by $$ \|T\|_{\alpha}= \sup \left\{ \sqrt{\alpha |\langle Tx,x \rangle|^2+ (1-\alpha)\|Tx\|^2 } : x\in \mathcal{H}, \|x\|=1 \right\} .$$ 
 This $\alpha$-norm is also a norm on $\mathcal{B}(\mathcal{H})$, and it satisfies 
 \begin{eqnarray}\label{eqv2}
 	w(T) \leq \|T\|_{\alpha} \leq \|T\|, \,\mbox{ for every $T \in  \mathcal{B}(\mathcal{H})$ and for all $\alpha\in [0,1]$}.
 \end{eqnarray}
 The inequality \eqref{eqv2} together with \eqref{eqv1} implies that the $\alpha$-norm is equivalent to both the operator norm and the numerical radius norm. For more details of the $\alpha$-norm, the reader can see \cite{Sain_AFA_2021}. Next, we turn our attention to study the $(\alpha,\beta)$-normal operators (see \cite{DM2008}), where $0\leq \alpha \leq 1 \leq \beta .$ An operator $T\in \mathcal{B}(\mathcal{H})$ is said to be $(\alpha,\beta)$-normal operator if $T$ satisfies $\alpha^2 T^*T\leq TT^*\leq \beta^2 T^*T.$ This is equivalent to 
 $ \alpha^2 \langle T^*Tx,x \rangle \leq \langle TT^*x,x \rangle \leq \beta^2 \langle T^*Tx,x \rangle \,\, \mbox{for all $x\in \mathcal{H}$},$
 that is,
 $$ \alpha \|Tx\| \leq \|T^*x\| \leq \beta \|Tx\| \,\, \mbox{for all $x\in \mathcal{H}$}.$$
 Every normal and hyponormal operators are $(\alpha,\beta)$-normal  for suitable values of $\alpha, \beta.$ For normal operators $\alpha=\beta=1$ and for hyponormal operators $\beta=1.$ Also, there exist operators which are neither normal nor hyponormal. As for example, the matrix $\begin{pmatrix}
 	1&0\\
 	1&1
 \end{pmatrix}$  is an $(\alpha,\beta)$-normal with $\alpha= \sqrt{\frac{3-\sqrt{5}}{2}}$ and $\beta= \sqrt{\frac{3+\sqrt{5}}{2}}.$  However, the matrix is neither normal nor hyponormal.  
 If $T$ is an  $(\alpha,\beta)$-normal operator, then both, $T$ mazorizes $T^*$ and $T^*$ mazorizes $T$. According to Douglas (Mazorization lemma) \cite{Douglas}, an operator $T\in \mathcal{B}(\mathcal{H})$ mazorizes an operator $S\in \mathcal{B}(\mathcal{H})$ if one of the following equivalent statements holds: \\
 (i) $Ran(T) \subseteq Ran(S)$, where $Ran(T)$ ($Ran(S)$) denotes the range space of $T$ ($S$).\\ (ii) $TT^*\leq \mu^2 SS^*$ for some $\mu\geq 0.$\\
 Therefore, $T$ is an  $(\alpha,\beta)$-normal operator if and only if $Ran(T) = Ran(T^*).$ So, every invertible operator is an $(\alpha,\beta)$-normal operator for suitable values of $\alpha,\beta$.

 \noindent 
 
 In this paper, we develop upper bounds for the $\alpha$-norm of bounded linear operators and consequentially we obtain bounds for the numerical radius,  which improve on the existing ones. Further, we develop lower bounds for the numerical radius of the $(\alpha,\beta)$-normal operators, which improve on the classical bound $w(T)\geq \frac{\|T\|}2.$

\section{\textbf{ The $\alpha$-norm and numerical radius inequalities }}
 
\noindent We obtain  new upper bounds for the numerical radius of bounded linear operators, by developing the bounds of $\alpha$-norm. First, we recall some necessary inequalities from the literature. 
Kato \cite{KAT52} generalized the Cauchy-Schwarz inequality, namely,  
\begin{equation}	\label{kato}
	\left\vert \left\langle Tx,y\right\rangle \right\vert ^{2}\leq \left\langle
	\left\vert T\right\vert ^{2\alpha }x,x\right\rangle \left\langle \left\vert
	T^{\ast }\right\vert ^{2\left(1-\alpha \right) }y,y\right\rangle\,\,\mbox{$\forall x,y\in \mathcal{H}$, $\forall \alpha \in \left[ 0,1\right]$.}
\end{equation}
 In particular, for $\alpha=1/2$ (see \cite[pp. 75-76]{HAL}), 
\begin{align}\label{kato1}
	|\langle Tx,y\rangle|\leq \langle |T|x,x\rangle^{1/2}~~\langle |T^*|y,y\rangle^{1/2} \,\,\, \forall x,y\in \mathcal{H}.
\end{align}
In \cite[Theorem 1]{KIT88}, Kittaneh generalized the Kato's inequality, which is   as follows:
If $f,g: [0,\infty]\to [0,\infty]$ are continuous functions satisfying $f(t)g(t)=t, ~~\forall t\geq 0$, then
	\begin{eqnarray}\label{kato2}
		\left |\langle Tx,y \rangle \right|^2\leq \left \langle f^2(|T|)x,x \right \rangle \left \langle g^2(|T^*|)y,y \right \rangle \,\,\, 	\forall  x,y\in \mathcal{H}.
	\end{eqnarray}
 Next, we recall the H\"{o}lder--McCarthy inequality (see \cite[p. 20]{SIM}) which states that 
	if $T\in \mathcal{B}(\mathcal{H})$ is positive,   then 
	\begin{eqnarray}\label{mc1}
		\langle Tx,x\rangle^r\leq \langle T^rx,x\rangle \quad \forall r\geq 1 \,\,\, \forall x \in \mathcal{H}, \|x\|=1.
	\end{eqnarray}
The inequality (\ref{mc1}) is reversed when $0\leq r\leq 1$. The Cartesian decomposition of $T\in \mathcal{B}(\mathcal{H})$ is given by $T=\Re(T)+i \Im(T)$, where $\Re(T)=\frac12(T+T^*)$ and $\Im(T)=\frac1{2i}(T-T^*)$. 
Now, we are in a position to prove our first result.

\begin{theorem}\label{th1}
	If $T\in \mathcal{B}(\mathcal{H})$, then
	$$ \|T\|_{\alpha}^2 \leq \left \|\frac{\alpha}4\left( f^4(|T|)+g^4(|T^*|)  \right)+(1-\alpha)|T|^2   \right\|+ \frac{\alpha}2 \left \|\Re \left(f^2(|T|)g^2(|T^*|) \right)  \right\|,$$
	where $f$ and $g$ are as in \eqref{kato2}.
\end{theorem}

\begin{proof}
	Let $x\in \mathcal{H}$ with $\|x\|=1.$ Then,  it follows from \eqref{kato2} that 
		\begin{eqnarray*}
		\left |\langle Tx,x \rangle \right|^2 &\leq& \left \langle f^2(|T|)x,x \right \rangle \left \langle g^2(|T^*|)x,x \right \rangle\\
		& =& \left ( \left \langle f^2(|T|)x,x \right \rangle^{1/2} \left \langle g^2(|T^*|)x,x \right \rangle^{1/2}\right )^2\\
		&\leq & \frac14 \left ( \left \langle f^2(|T|)x,x \right \rangle+ \left \langle g^2(|T^*|)x,x \right \rangle\right )^2\\
		&= & \frac14  \left \langle \left ( f^2(|T|)+ g^2(|T^*|)\right )x,x \right \rangle^2\\
		&\leq & \frac14  \left \langle \left ( f^2(|T|)+ g^2(|T^*|)\right )^2 x,x \right \rangle \,\,\, (\mbox{by using  \eqref{mc1} } )\\
		&=& \frac14  \left \langle \left ( f^4(|T|)+ g^4(|T^*|)\right ) x,x \right \rangle + \frac12 \left \langle \Re \left(f^2(|T|)g^2(|T^*|) \right) x,x \right \rangle.
	\end{eqnarray*}
Thus, we have
\begin{eqnarray*}
	\|T\|_{\alpha}^2 &=& \sup_{\|x\|=1} \left \{ \alpha \left |\langle Tx,x \rangle \right|^2+ (1-\alpha)\|Tx\|^2 \right \}\\
	 &\leq &  \sup_{\|x\|=1} \left \langle \left( \frac \alpha 4\left ( f^4(|T|)+ g^4(|T^*|)\right )+ (1-\alpha)|T|^2 \right) x,x \right \rangle\\
	  && + \frac\alpha2 \sup_{\|x\|=1} \left |\left \langle \Re \left(f^2(|T|)g^2(|T^*|) \right) x,x \right \rangle \right|\\
	& =&  \left \| \frac \alpha 4\left ( f^4(|T|)+ g^4(|T^*|)\right )+ (1-\alpha)|T|^2 \right \| + \frac\alpha2 \left \|\Re \left(f^2(|T|)g^2(|T^*|) \right) \right\|,
\end{eqnarray*}
as required.
\end{proof}

Now, as a consequence of Theorem \ref{th1}, we get  the following corollary.

\begin{cor}\label{cor1}
		If $T\in \mathcal{B}(\mathcal{H})$, then
		\begin{eqnarray}\label{eqn1}
			\|T\|_{\alpha}^2 &\leq & \left \| \left (1-\frac{3\alpha}4 \right) |T|^2+ \frac \alpha 4 |T^*|^2 \right \|+ \frac \alpha 2 \left \| \Re (|T||T^*|)\right \|
		\end{eqnarray}
		  and 
		  \begin{eqnarray}\label{eqn2}
		  	\|T\|_{\alpha}^2 &\leq & \left \| \left (1-\frac{3\alpha}4 \right) |T^*|^2+ \frac \alpha 4 |T|^2 \right \|+ \frac \alpha 2 \left \| \Re (|T||T^*|)\right \|.
		  \end{eqnarray}
	  Further, 
	  \begin{eqnarray}\label{cor1_2}
	   w(T) \leq \min \{ \gamma, \delta\},
\end{eqnarray} 
  \mbox{where}
	  $$ \gamma= \min_{\alpha \in [0,1]} \left \{ \left \| \left (1-\frac{3\alpha}4 \right) |T|^2+ \frac \alpha 4 |T^*|^2 \right \|+ \frac \alpha 2 \left \| \Re (|T||T^*|)\right \|  \right \}^{1/2} $$
	  and 
	  $$ \delta= \min_{\alpha \in [0,1]} \left\{ \left \| \left (1-\frac{3\alpha}4 \right) |T^*|^2+ \frac \alpha 4 |T|^2 \right \|+ \frac \alpha 2 \left \| \Re (|T||T^*|)\right \| \right\}^{1/2}.$$
\end{cor}
\begin{proof}
	The inequality \eqref{eqn1} follows from Theorem \ref{th1} by taking $f(t)=g(t)=t^{1/2}$. The inequality \eqref{eqn2} follows from \eqref{eqn1} by replacing  $T$ by $T^*$ and using the equality $\|T\|_{\alpha}=\|T^*\|_{\alpha}$ (see \cite[Proposition  2.6]{Sain_AFA_2021}). The inequality \eqref{cor1_2} follows from the fact that $w(T) \leq \min_{\alpha \in [0,1]} \|T\|_{\alpha}$.
\end{proof}

\begin{remark}\label{rem1} Recently, Bhunia and Paul obtained  in \cite[The inequality (2.3)]{Bhunia_ASM_2022} and \cite[Corollary 2.6]{Bhunia_BSM_2021}  that  
	$w^2(T) \leq \frac14  \left \| |T|^2+ |T^*|^2 \right \|+ \frac12  \left \| \Re (|T||T^*|)\right \|$
	and 
	$ w^2(T) \leq  \frac14  \left \| |T|^2+ |T^*|^2 \right \|+ \frac12  w (|T||T^*|),$ respectively.
	Clearly, we see that 
	\begin{eqnarray*} 
		\min \{ \gamma^2, \delta^2\} &\leq &\frac14  \left \| |T|^2+ |T^*|^2 \right \|+ \frac12  \left \| \Re (|T||T^*|)\right \|\\
		& \leq & \frac14  \left \| |T|^2+ |T^*|^2 \right \|+ \frac12  w (|T||T^*|).
	\end{eqnarray*}
	Thus, we would like to remark that the inequality \eqref{cor1_2} refines both the inequalities in \cite[The inequality (2.3)]{Bhunia_ASM_2022} and \cite[Corollary 2.6]{Bhunia_BSM_2021}. To show proper refinement, we consider a matrix $T=\begin{pmatrix}
		0&1&0\\
		0&0&2\\
		0&0&0
	\end{pmatrix}$. Then, by elementary calculations we see that $\gamma=\sqrt{\frac{28}{13}}$ and $\delta=\frac32$, and so
$$ \min \{ \gamma^2, \delta^2\}= \frac{28}{13} < \frac{9}{4}= \frac14  \left \| |T|^2+ |T^*|^2 \right \|+ \frac12  \left \| \Re (|T||T^*|)\right \|.$$
\end{remark}


Another bound for the $\alpha$-norm reads as follows:
	
	\begin{theorem}\label{th2}
		If $T\in \mathcal{B}(\mathcal{H})$, then
		\begin{eqnarray*}
			\|T\|_{\alpha}^2 &\leq & \min \left\{ \left \| \alpha |T|^2+ (1-\alpha)|T^*|^2 \right \|, \left \| \alpha |T^*|^2+ (1-\alpha)|T|^2 \right \| \right\}.
		\end{eqnarray*} 
	\end{theorem}

\begin{proof}
	Let $x\in \mathcal{H}$ with $\|x\|=1.$ By using Cauchy-Schwarz inequality, we get
	 \begin{eqnarray}\label{eqn3}
		\|T\|_{\alpha}^2 &=&  \sup_{\|x\|=1} \left \{ \alpha \left |\langle Tx,x \rangle \right|^2+ (1-\alpha)\|Tx\|^2 \right \}\nonumber\\
		&=&  \sup_{\|x\|=1} \left \{ \alpha \left |\langle x,T^*x \rangle \right|^2+ (1-\alpha)\|Tx\|^2 \right \}\nonumber\\
		&\leq& \sup_{\|x\|=1} \left \{ \alpha \|T^*x\|^2+ (1-\alpha)\|Tx\|^2 \right \}\nonumber\\
			&=& \sup_{\|x\|=1} \left \{ \alpha \langle |T^*|^2x,x\rangle + (1-\alpha)\langle |T|^2x,x \rangle  \right \}\nonumber\\
			&=& \sup_{\|x\|=1}  \langle (\alpha |T^*|^2+ (1-\alpha) |T|^2)x,x\rangle \nonumber\\
			&=& \| \alpha |T^*|^2+ (1-\alpha) |T|^2 \|.
	\end{eqnarray}
Replacing $T$ by $T^*$ in \eqref{eqn3} and using $\|T\|_{\alpha}=\|T^*\|_{\alpha}$, we get
\begin{eqnarray}\label{eqn4}
	\|T\|_{\alpha}^2  & \leq & \| \alpha |T|^2+ (1-\alpha) |T^*|^2 \|.
\end{eqnarray}
Therefore, the required inequality follows by combining \eqref{eqn3} and \eqref{eqn4}.
\end{proof}

Now, the following upper bound for the numerical radius follows easily  from Theorem \ref{th2} together with the first inequality in \eqref{eqv2}:
\begin{eqnarray}\label{pp0}
	w^2(T) &\leq & \min_{\alpha \in [0,1]} \left\| \alpha |T|^2+(1-\alpha)|T^*|^2 \right\|.
\end{eqnarray}
	Note that, the bound (\ref{pp0}) is  also obtained in \cite{Bhunia_RIM_2021}. 
	To obtain our next result we need the Buzano's inequality \cite{Buzano}, namely, if $x,y,e\in \mathcal{H}$ with $\|e\|=1,$ then
	\begin{eqnarray}\label{buza}
		2 |\langle x,e\rangle \langle e,y\rangle| &\leq & \|x\| \|y\|+ |\langle x,y\rangle|.
	\end{eqnarray}
By employing the Buzano's inequality \eqref{buza}, we prove the following theorem.
\begin{theorem}\label{th3}
	If $T\in \mathcal{B}(\mathcal{H})$, then
	\begin{eqnarray*}
		\|T\|_{\alpha}^2 & \leq & \frac \alpha 4 w^2(|T|+i |T^*|)+ \frac \alpha 4 w(|T| |T^*|)+ \left \|  \left(1-\frac{7 \alpha}{8} \right)|T|^2+ \frac \alpha 8 |T^*|^2 \right \|.
	\end{eqnarray*}
\end{theorem}
	\begin{proof}
		Let $x\in \mathcal{H}$ with $\|x\|=1.$ Then, from the inequality \eqref{kato1}, we  get
		\begin{eqnarray*}
			&& |\langle Tx,x\rangle|^2 \\
			&\leq & \langle |T|x,x\rangle  \langle |T^*|x,x\rangle\\
			& \leq & \frac14 \left( \langle |T|x,x\rangle + \langle |T^*|x,x\rangle \right)^2\\
			&=& \frac14 (\langle |T|x,x\rangle^2 + \langle |T^*|x,x\rangle^2) + \frac12 \langle |T|x,x\rangle  \langle |T^*|x,x\rangle\\
			&=& \frac14 |\langle |T|x,x\rangle +i \langle |T^*|x,x\rangle|^2+ \frac12 \langle |T^*|x,x\rangle  \langle x,|T|x\rangle\\
			&\leq & \frac14 |\langle (|T|+i  |T^*|)x,x\rangle|^2+ \frac14 \| |T^*|x\| \,\| |T|x \| + \frac14 |\langle |T^*|x, |T|x \rangle| \,\,\, (\mbox{by \eqref{buza}})\\
			& \leq & \frac14 |\langle (|T|+i  |T^*|)x,x\rangle|^2+ \frac18 ( \| |T^*|x\|^2 +\| |T|x \|^2) + \frac14 |\langle |T||T^*|x, x \rangle|.
		\end{eqnarray*}
	Hence, 
\begin{eqnarray*}
&&	\alpha |\langle Tx,x\rangle|^2 	+ (1-\alpha) \|Tx\|^2\\
	& \leq & \frac \alpha 4 |\langle (|T|+i  |T^*|)x,x\rangle|^2
	 + \frac \alpha 4 |\langle |T||T^*|x, x \rangle|
	  +  \left \langle  \left( \left( 1-\frac {7\alpha}8 \right)   |T|^2+ \frac \alpha 8 |T^*|^2  \right) x,x \right \rangle \\
	 & \leq & \frac \alpha 4 w^2 (|T|+i  |T^*|) + \frac \alpha 4 w(|T||T^*|) 
	 + \left \| \left( 1-\frac {7\alpha}8 \right)   |T|^2+ \frac \alpha 8 |T^*|^2  \right \|.
	\end{eqnarray*}
Therefore, taking supremum over $\|x\|=1$, we get the required inequality.
	\end{proof}

Replacing $T$ by $T^*$ in Theorem \ref{th3}, and then using $\|T\|_{\alpha}=\|T^*\|_{\alpha}$, we also get the following inequality:
\begin{cor}\label{cor3}
	If $T\in \mathcal{B}(\mathcal{H})$, then
	\begin{eqnarray*}
		\|T\|_{\alpha}^2 & \leq & \frac \alpha 4 w^2(|T|+i |T^*|)+ \frac \alpha 4 w(|T| |T^*|)+ \left \|  \left(1-\frac{7 \alpha}{8} \right)|T^*|^2+ \frac \alpha 8 |T|^2 \right \|.
	\end{eqnarray*}
\end{cor}

As an application of the bounds for the $\alpha$-norm, obtained in Theorem \ref{th3} and Corollary \ref{cor3}, we have a new upper bound for the numerical radius, which is given below.
\begin{cor}\label{cor4}
	If $T\in \mathcal{B}(\mathcal{H})$, then
	\begin{eqnarray*}
		w^2(T) & \leq & \frac \alpha 4 w^2(|T|+i |T^*|)+ \frac \alpha 4 w(|T| |T^*|) \\
		&& +\min \left \{ \left \|  \left(1-\frac{7 \alpha}{8} \right)|T^*|^2+  \frac \alpha 8 |T|^2 \right \|, \left \|  \left(1-\frac{7 \alpha}{8} \right)|T|^2+ \frac \alpha 8 |T^*|^2 \right \| \right \},
	\end{eqnarray*}
for all $\alpha \in [0,1].$
	\end{cor}
	In particular, for $\alpha=1$ we have (see also in \cite{Jana_2022}),
	\begin{eqnarray}\label{eqn5}
		w^2(T) & \leq & \frac 1 4 w^2(|T|+i |T^*|)+ \frac 1 4 w(|T| |T^*|) 
		+\frac 1 8\left \|  |T^*|^2+   |T|^2 \right \|\\
		& \leq & \frac 12 \left \|  T^*T+TT^* \right \|\nonumber.
	\end{eqnarray}
	 Therefore, Corollary \ref{cor4} generalizes and improves on the inequality \eqref{eqn5} as well as the well-known inequality $w^2(T)\leq \frac 12 \left \|  T^*T+TT^* \right \|$, obtained by Kittaneh \cite{Kittaneh_STD_2005}.

	  Next, we note that every operator  $T\in \mathcal{B}(\mathcal{H})$  satisfies the following numerical radius inequality (see \cite{Kittaneh_2003}): $w(T)\leq \frac12 \| |T|+|T^*| \|.$ In \cite{Bhunia_RIM_2021}, it is  proved that $w^2(T) \leq  \left\| \alpha |T|^2+(1-\alpha)|T^*|^2 \right\|, \,\, \forall  \alpha \in [0,1].$
	 Now, it is natural to ask whether the following inequality  $$w(T) \leq  \left\| \alpha |T|+(1-\alpha)|T^*| \right\|, \,\, \mbox{$\forall \alpha \in [0,1]$}$$ holds or not.
In this connection we have the following result.

	
	\begin{theorem}\label{th4}
		If $T\in \mathcal{B}(\mathcal{H}),$ then
		\begin{eqnarray*}
			w^2(T) &\leq & \left(\left\| \alpha |T|+(1-\alpha)|T^*| \right\|\right)  \|T\|,
		\end{eqnarray*}
	for all $\alpha \in [0,1].$
	\end{theorem}
	\begin{proof}
		Let $x\in \mathcal{H}$ with $\|x\|=1.$ Then, it follows from the inequality \eqref{kato1} that
		\begin{eqnarray*}
			|\langle Tx,x\rangle|^2 &\leq & \langle |T|x,x\rangle \, \langle |T^*|x,x\rangle\\
			&=& \langle |T|x,x\rangle^{\alpha} \langle |T^*|x,x\rangle^{1-\alpha} \langle |T|x,x\rangle^{1-\alpha}  \langle |T^*|x,x\rangle^{\alpha}\\
			&\leq & (\alpha \langle |T|x,x\rangle+ (1-\alpha)\langle |T^*|x,x\rangle) ((1-\alpha)\langle |T|x,x\rangle+ \alpha \langle |T^*|x,x\rangle  )\\
			&=& \langle (\alpha |T| + (1-\alpha)|T^*|)x,x\rangle  \langle (\alpha |T^*| + (1-\alpha)|T|)x,x\rangle\\
			& \leq & \left(\| \alpha |T| + (1-\alpha)|T^*| \|\right)  \,\, \left(\| \alpha |T^*| + (1-\alpha)|T|\|\right)\\
			 & \leq & \left(\| \alpha |T| + (1-\alpha)|T^*| \|\right)   \|T\|.
		\end{eqnarray*}
	Therefore, taking supremum over $\|x\|=1$, we get 
	\begin{eqnarray*}
		w^2(T) &\leq & \left(\left\| \alpha |T|+(1-\alpha)|T^*| \right\|\right)  \|T\|,
	\end{eqnarray*}
as required.
	\end{proof}

	\begin{remark}
	Since Theorem \ref{th4} holds for all $\alpha\in [0,1]$, we can write that inequality in the following form:
	\begin{eqnarray}\label{impr1}
		w(T) &\leq & \sqrt{\left( \min_{\alpha \in [0,1]}\left\| \alpha |T|+(1-\alpha)|T^*| \right\| \right) \|T\|}.
	\end{eqnarray}
	Note that, the bound \eqref{impr1} is a considerable refinement of the well-known bound $w(T)\leq \|T\|$. To show proper refinement, we take the same matrix as in Remark \ref{rem1}. Then, by elementary calculations we see that 
$$\sqrt{\left( \min_{\alpha \in [0,1]}\left\| \alpha |T|+(1-\alpha)|T^*| \right\| \right) \|T\|}=\sqrt{\frac43 \times 2} < 2=\|T\|.$$
\end{remark}


 \section{\textbf{Numerical radius of $(\alpha, \beta)$-normal operators}}
 \noindent In this section, we develop lower bounds for the numerical radius of the $(\alpha, \beta)$-normal operators.
 We observe that for  $T\in \mathcal{B}(\mathcal{H})$,
 \begin{eqnarray}\label{low1}
 	w(T)\geq \max\{ \|\Re(T)\|, \|\Im(T)\|\},
 \end{eqnarray}
  the proof of which follows from the Cartesian decomposition of $T.$ The  first lower bound of this section reads as follows: 
 
 \begin{theorem}\label{th5}
 	Let $T$ be an $(\alpha,\beta)$-normal operator. Then
 	$$ w^2(T)\geq \max \left\{ 1+\alpha^2, 1+\frac{1}{\beta^2}\right\} \frac{\|T\|^2}{4}+ \frac {\left| \|\Re(T)\|^2-\|\Im(T)\|^2  \right|}2 .$$
 \end{theorem}
 \begin{proof}
 	From the inequality \eqref{low1}, we have 
 	\begin{eqnarray}\label{low2}
 		w^2(T) &\geq &\frac{\| \Re(T)\|^2+ \|\Im(T)\|^2}{2} + \frac{ | \| \Re(T)\|^2- \|\Im(T)\|^2|}{2} \nonumber \\
 		&\geq& \frac{\| (\Re(T) )^2 + (\Im(T) )^2 \|}{2} + \frac{| \| \Re(T)\|^2- \|\Im(T)\|^2|}{2} \nonumber\\
 		&=& \frac{\| |T|^2+|T^*|^2 \|}{4} + \frac{| \| \Re(T)\|^2- \|\Im(T)\|^2| }{2} \nonumber\\
 		&\geq& \frac{\| |T|^2+\alpha^2 |T|^2 \|}{4} + \frac{| \| \Re(T)\|^2- \|\Im(T)\|^2| }{2} \,\,\,(\mbox{since, $|T^*|^2\geq \alpha^2 |T|^2$}) \nonumber\\
 		&=& \frac{(1+\alpha^2) \|T\|^2}{4} + \frac{| \| \Re(T)\|^2- \|\Im(T)\|^2|}{2}.
 	\end{eqnarray}
 Similarly, we also have
 \begin{eqnarray}\label{low3}
 	w^2(T) &\geq & \frac{\left (1+\frac 1 {\beta^2} \right) \|T\|^2}{4} + \frac{| \| \Re(T)\|^2- \|\Im(T)\|^2|}{2}.
 \end{eqnarray}
Therefore, the required inequality follows by combining the inequalities in \eqref{low2} and \eqref{low3}.
 \end{proof}

 \begin{remark} 
 	Clearly, $\max \left\{ 1+\alpha^2, 1+\frac{1}{\beta^2}\right\}> 1.$ Therefore,  for an  $(\alpha,\beta)$-normal operator $T$,
 	\begin{eqnarray}\label{sab}
 	w(T)	&\geq & \sqrt{\max \left\{ 1+\alpha^2, 1+\frac{1}{\beta^2}\right\} \frac{\|T\|^2}{4}+ \frac {\left| \|\Re(T)\|^2-\|\Im(T)\|^2  \right|}2} \nonumber \\
 		&\geq & \sqrt{\max \left\{ 1+\alpha^2, 1+\frac{1}{\beta^2}\right\} \frac{\|T\|^2}{4} } \nonumber\\
 		&=& \max \left\{ \sqrt{1+\alpha^2}, \sqrt{1+\frac{1}{\beta^2}} \right\} \frac{\|T\|}{2}  \\
 		&> & \frac{\|T\|}2 \nonumber.
 	\end{eqnarray}
 Therefore, we would like to remark that for the $(\alpha,\beta)$-normal operators, Theorem \ref{th5} gives  stronger bound than the well-known bound $w(T)\geq \frac{\|T\|}{2}$. Note that the inequality \eqref{sab} is also developed in \cite[Theorem 3.1]{SAB}.
 	\end{remark}
 
 To obtain another improvement we first note the following numerical radius inequality (see \cite{Bhunia_arxiv}): For every $T\in \mathcal{B}(\mathcal{H})$,
 \begin{eqnarray}\label{low4}
 	w(T)\geq \frac 1 {\sqrt{2}} \max \left \{ \|\Re(T)+\Im(T)\|,  \|\Re(T)-\Im(T)\| \right\}.
 \end{eqnarray} 
By using the inequality \eqref{low4} and using the similar arguments as Theorem \ref{th5}, we also prove the following lower bound for the numerical radius of  the $(\alpha,\beta)$-normal operators.

\begin{theorem}\label{th6}
	Let $T$ be an $(\alpha,\beta)$-normal operator. Then
	$$ w^2(T)\geq \max \left\{ 1+\alpha^2, 1+\frac{1}{\beta^2}\right\} \frac{\|T\|^2}{4}+ \frac {\left| \|\Re(T)+ \Im(T)\|^2 - \|\Re(T)- \Im(T)\|^2 \right|}4 .$$
	
\end{theorem}

\begin{remark}
	Clearly, we see that for an $(\alpha,\beta)$-normal operator $T$,
	 \begin{eqnarray*}
	w(T)	&\geq & \sqrt{\max \left\{ 1+\alpha^2, 1+\frac{1}{\beta^2}\right\} \frac{\|T\|^2}{4}+ \frac {\left| \|\Re(T)+ \Im(T)\|^2 - \|\Re(T)- \Im(T)\|^2 \right|}4} \\
	&\geq & \sqrt{\max \left\{ 1+\alpha^2, 1+\frac{1}{\beta^2}\right\} \frac{\|T\|^2}{4} }\\
	&=& \max \left\{ \sqrt{1+\alpha^2}, \sqrt{1+\frac{1}{\beta^2}} \right\} \frac{\|T\|}{2}  \\
	&> & \frac{\|T\|}2.
\end{eqnarray*}
Therefore,  for the  $(\alpha,\beta)$-normal operators, Theorem \ref{th6} gives  stronger bound than the well-known bound $w(T)\geq \frac{\|T\|}{2}$. 
\end{remark}

At the end, we would like to remark that if  $T\in \mathcal{B}(\mathcal{H})$ is an invertible operator, then $w(T)> \frac{\|T\|}{2}.$ This implies that when $w(T)=\frac {\|T\|}2$, then $T$ must be a non-invertible operator. 

\bibliographystyle{amsplain}

\begin{thebibliography}{99}
	
\bibitem{Bag_MIA_2020} S. Bag, P. Bhunia,  and K. Paul; Bounds of numerical radius of bounded linear operator using $t$-Aluthge transform, Math. Inequal. Appl.  23 (2020), no. 3, 991--1004.	
	
\bibitem{book} P. Bhunia, S.S. Dragomir, M.S. Moslehian, and K. Paul; Lectures on numerical radius inequalities, Infosys Science Foundation Series, Infosys Science Foundation Series in Mathematical Sciences, Springer Cham, (2022). https://doi.org/10.1007/978-3-031-13670-2
	
	\bibitem{Bhunia_ASM_2022} P. Bhunia, and K. Paul; Refinement of numerical radius inequalities of complex Hilbert space operators, {Acta Sci. Math. (Szeged)} (2022), to appear.
	
	\bibitem{Bhunia_MMN_2022} P. Bhunia, A. Bhanja, D. Sain, and K. Paul;  Numerical radius inequalities of operator matrices from a new norm on $\mathcal{B}(\mathcal{H})$, Miskolc Math. Notes (2022), to appear.
	
	\bibitem{Bhunia_JCA_2022} P. Bhunia, A. Sen, and K. Paul;  New semi-norm of semi-Hilbertian space operators and its application,  J. Convex Anal. 29 (2022), no. 4, 1149--1160.
	
	\bibitem{Bhunia_LAA_2021} P. Bhunia, and K. Paul;  Development of inequalities and characterization of equality conditions for the numerical radius, Linear Algebra Appl. 630 (2021), 306--315. 
	
	\bibitem{Bhunia_RIM_2021} P. Bhunia, and K. Paul; Proper improvement of well-known numerical radius inequalities and their applications, Results Math. 76 (2021), no. 4, Paper No. 177, 12 pp. 
	
	\bibitem{Bhunia_BSM_2021} P. Bhunia, K. Paul; New upper bounds for the numerical radius of Hilbert space operators, Bull. Sci. Math. 167 (2021), Paper No. 102959, 11 pp. 
	
		\bibitem{Bhunia_arxiv} P. Bhunia, S. Jana, and K. Paul; Refined inequalities for the numerical radius of Hilbert space operators, (2021). https://arxiv.org/abs/2106.13949
	
	\bibitem{Buzano} M.L. Buzano; Generalizzatione della disuguaglianza di Cauchy-Schwarz, Rend. Semin. Mat. Univ. Politech. Torino 31(1971/73) (1974) 405--409.
	
	
	\bibitem{DM2008} S.S. Dragomir, and M.S. Moslehian; Some inequalities for $(\alpha, \beta)$-normal operators in Hilbert spaces,  Facta Univ. Ser. Math. Inform. 23 (2008), 39--47.
	
	\bibitem{Douglas} R.G. Duglas; On majorization, factorization, and range inclusion of operators on Hilbert space, Proc. Amer. Math. Soc. 17 (1966), 413--415.
	
	
	
		\bibitem{HAL} P.R. Halmos; A Hilbert space problems book, Springer Verlag, New York, 1982.
		
		
		\bibitem{Jana_2022} S. Jana, P. Bhunia, and K. Paul; Numerical radius inequalities and estimation of zeros of polynomials, (2023).
		https://arxiv.org/abs/2301.03159
	
	\bibitem{KAT52} T. Kato; Notes on some inequalities for linear operators, \textit{Math. Ann.} {125} (1952), 208--212.
	
	\bibitem{Kittaneh_STD_2005}  F. Kittaneh; Numerical radius inequalities for Hilbert space operators, Studia Math. 168 (2005), no. 1, 73--80.
	
	\bibitem{Kittaneh_2003} F. Kittaneh; Numerical radius inequality and an estimate for the numerical radius of the Frobenius companion matrix, Studia Math. 158 (2003), no. 1, 11--17.
	
	
	
	\bibitem{KIT88} F. Kittaneh; Notes on some inequalities for Hilbert space operators, \textit{Publ. RIMS Kyoto Univ.} {24} (1988) 283--293.
	
	
	\bibitem{SAB} M. Sababheh, and H.R. Moradi;  New orders among Hilbert space operators, (2022). https://arxiv.org/abs/2212.06382v1
	
	\bibitem{Sain_AFA_2021} D. Sain, P. Bhunia, A. Bhanja, and K. Paul; On a new norm on $\mathcal{B}(\mathcal{H})$ and its application to numerical radius inequalities,  Ann. Funct. Anal. 12 (2021), no. 4, Paper No. 51, 25 pp. 
	

	
\bibitem{SIM} B. Simon; Trace ideals and their applications, Camrbidge University Press, 1979.
	
	
	
	
	
	


\end{thebibliography}

\end{document}